\documentclass[12pt]{amsproc}
\newtheorem{theorem}{\sc Theorem}[section]
\newtheorem{lemma}[theorem]{\sc Lemma}

\newtheorem{corollary}[theorem]{\sc Corollary}

\newcommand{\zp}{\mathbb{Z}_p}

\newcommand{\hz}{\widehat{\mathbb{Z}}}

\begin{document}

\author{Jo\~ao Azevedo}
\address{Department of Mathematics, University of Brasilia\\
Brasilia-DF \\ 70910-900 Brazil}
\email{J.P.P.Azevedo@mat.unb.br}

\author{Pavel Shumyatsky}
\address{Department of Mathematics, University of Brasilia\\
Brasilia-DF \\ 70910-900 Brazil}
\email{pavel@unb.br}
\thanks{Supported by CNPq and FAPDF}

\title[Conciseness of some words]{On finiteness of some verbal subgroups in profinite groups}
 \begin{abstract} Given a group word $w$ and a group $G$, the set of $w$-values in $G$ is denoted by $G_w$ and the verbal subgroup $w(G)$ is the one generated by $G_w$. In the present paper we consider profinite groups admitting a word $w$ such that the cardinality of $G_w$ is less than $2^{\aleph_0}$ and $w(G)$ is generated by finitely many $w$-values. For several families of words $w$ we show that under these assumptions $w(G)$ must be finite. Our results are related to the concept of conciseness of group words.
\end{abstract}
\maketitle

\section{Introduction}

Let $w = w(x_1,\ldots,x_{k})$ be a group word, i.e., a nontrivial element of the free group on $x_1, \ldots, x_{k}$.  This can be viewed as a function of $k$ variables defined on any group $G$. We denote by $G_w$ the set of all $w$-values in $G$ and  by $w(G)$ the verbal subgroup generated by $G_w$. In the context of topological groups $G$, we write $w(G)$ to denote
the closed subgroup generated by all $w$-values in $G$.

The word $w$ is called concise in a class $\mathcal{C}$ of groups if, for each $G$ in $\mathcal{C}$ such that $G_w$ is finite, also $w(G)$ is finite.  In the sixties Hall raised the problem whether all  words are concise. In 1989 S. Ivanov \cite{ivanov} (see also \cite[p.\ 439]{ols}) solved the problem in the negative. On the other hand, the problem for residually finite and profinite groups remains open (cf. Segal \cite[p.\ 15]{Segal} or Jaikin-Zapirain \cite{jaikin}). In recent years some limited progress with respect to this problem was made. In particular, it was shown in \cite{as} that if $w$ is a multilinear commutator word and $n$ is a prime-power, then the word $w^n$ is concise in the class of residually finite groups. Further examples of words that are concise in residually finite groups were discovered in \cite{gushu}. The works \cite{dms1,dms2,dms3} deal with conciseness of words of Engel type.

For profinite groups a variation of the classical notion of conciseness arises quite naturally: following \cite{dks} we say that $w$ is strongly concise in a class $\mathcal{C}$ of profinite groups if, for each $G$ in $\mathcal{C}$, already the bound $\lvert G_w \rvert < 2^{\aleph_0}$ implies that $w(G)$ is finite (a somewhat weaker notion of conciseness for profinite groups was considered in \cite{dms0}). It was conjectured in \cite{dks} that every group word is strongly concise, and the conjecture was confirmed for several families of group words. In particular, in \cite{dks} the conjecture was confirmed for multilinear commutator words (see the next section for the relevant definition). Moreover, it was shown that in the class of nilpotent profinite groups every group word is strongly concise.

In \cite{dks}, special attention was given to profinite groups admitting a word $w$ such that $\lvert G_w \rvert < 2^{\aleph_0}$ and $w(G)$ is generated by finitely many $w$-values. The finiteness of $w(G)$ under these assumptions was established for several types of words. In particular, $w(G)$ was shown to be finite whenever $w$ implies virtual nilpotency or is weakly
rational.

Recall that a group word $w$ is a law in a group $G$ if
$w(G)=1$.  We say that $w$ implies virtual nilpotency if every
finitely generated metabelian group for which $w$ is a law has a
nilpotent subgroup of finite index.  Burns and Medvedev~\cite{BuMe03}
showed that if $w$ implies virtual nilpotency, then for a much larger
class of groups~$G$, including all finitely generated residually
finite groups, $w(G) = 1$ implies that $G$ is nilpotent-by-finite. By a result of Gruenberg~\cite{Gr53}, all Engel words $[y,{}_n x]=[y,\underbrace{x,\dots,x}_n]$ imply virtual nilpotency. Other examples of words implying virtual nilpotency include generalisations of Engel words, such as words of the form $w = w(x,y) = [y^{e_1},x^{e_2},\ldots,x^{e_{k}}]$, where   $k\geq1$ and $e_1,\ldots,e_{k}$ are non-zero integers (see \cite[Section 4]{dks}). Henceforth, we use the left-normed simple commutator notation
$[x_1,x_2,x_3,\dots ,x_k]:=[...[[x_1,x_2],x_3],\dots ,x_k]$ and the abbreviation $[y,\,{}_nx]:=[y,x,\dots, x]$ where $x$ is repeated $n$ times. The word $$\gamma_k(x_1,x_2,\dots ,x_k)=[x_1,x_2,\dots ,x_k]$$ is called the $k$th lower central word and $\gamma_k(G)$ is of course the $k$th term of the lower central series of a group $G$. 

Following~\cite{gushu} we say that a group word $w$ is `weakly rational' if, for every finite group $G$ and for every positive integer $e$ with $\gcd(e,\lvert G \rvert) = 1$, the set $G_w$ is closed under taking $e$th powers of its elements. According to \cite[Theorem~3]{gushu}, the word $w={\gamma_k}^q$ is weakly rational for all positive integers $k,q$.

Our first result in the present paper is as follows.

\begin{theorem}\label{theorem 1}
Let $k,n$ and $q$ be positive integers, and let $v=\gamma_k(x_1,x_2,\dots ,x_k)$. Suppose that $w$ is one of the words $[y, {}_n v^q]$ or $[v^q, {}_n y]$. If $G$ is a profinite group such that $|G_w| < 2^{\aleph_0}$ and $w(G)$ is generated by finitely many $w$-values, then $w(G)$ is finite.
\end{theorem}

Note that \cite[Theorem 3]{dms1} implies that the word $w$ in Theorem \ref{theorem 1} is concise (in the usual sense) in the class of profinite groups. Our next result is related to \cite[Theorem 1.2]{dms3}.  Let us say that a word $w$ is commutator-closed if the set of $w$-values in any group is closed under taking commutators of its elements. It is easy to see that in particular the lower central words $\gamma_k$ are commutator-closed.

\begin{theorem}\label{theorem 2}
Let $n$ and $q$ be non-negative integers. Let $u$ be one of the words $y^q$, $[y_1, y_2]^q$ or any commutator-closed word. Assume that $v$ is a weakly rational word such that all $v$-values are also $u^{-1}$-values in any group and consider $w = [v, {}_n u]$. If $G$ is a profinite group such that $|G_w| < 2^{\aleph_0}$ and $w(G)$ is generated by finitely many $w$-values, then $w(G)$ is finite.
\end{theorem}

As mentioned in \cite{dms3} there are many words of the form $[v, {}_n u]$ as above. In particular, the word $[v,{}_n\,y]$, where $v=v(x_1,\dots,x_k)$ is any weakly rational word, is of the required shape. Other obvious examples of words for which Theorem \ref{theorem 2} applies are the words of the form $[v,{}_n\,u]$, where $v=[x_1^q,x_2,...,x_k]$ and $u=\gamma_l$ for $l\leq k$.

\section{Preliminary Results}

Throughout this paper by a subgroup of a profinite group we always mean a closed subgroup. We write $\langle S\rangle$ for the subgroup topologically generated by a set $S$.

Recall that a multilinear commutator word, also known as an outer-commutator word, is obtained  by nesting commutators and using each variable only once. Thus the word $[[x_1,x_2],[x_3,x_4,x_5],x_6]$ is a multilinear commutator word  while the $3$-Engel word $[x,y,y,y]$ is not.  An important family of multilinear commutator words consists of  the repeated commutator words $\gamma_k$ on $k$ variables,  given by $\gamma_1=x_1$ and $\gamma_k=[\gamma_{k-1},x_k]= [x_1,\ldots,x_k]$ for $k \ge 2$. As mentioned in the introduction, the verbal subgroup $\gamma_k(G)$ of a group $G$ is the $k$th term of the lower central series of~$G$. The derived words $\delta_k$, on $2^k$ variables, form another distinguished family of multilinear commutators; they are defined by $\delta_0=x_1$ and $\delta_k=[\delta_{k-1}(x_1,\ldots,x_{2^{k-1}}),
\delta_{k-1}(x_{2^{k-1}+1},\ldots,x_{2^k})]$. The verbal subgroup  $\delta_k(G)=G^{(k)}$ is the $k$th  derived subgroup of $G$.

The following lemma is taken from \cite[Lemma 2.2]{fernandez-morigi}.

\begin{lemma}\label{mcw symmetric}  Let $w$ be a multilinear commutator word and $G$ a group. Then $G_w$ is symmetric, that is, $x\in G_w$ implies that $x^{-1}\in G_w$.
\end{lemma} 

An element $a$ of a group $G$ is called a right Engel element if, for every $g\in G$, there exists $n=n(a,g)$ such that $[a,{}_n g] = 1$. If $n$ can be chosen independently of $g$, we say that the element $a$ is right $n$-Engel. If $[g, {}_n a] = 1$ holds for all $g \in G$ and some $n = n(a,g)$, the element $a$ is called a (left) Engel element, and we call it (left) $n$-Engel if $n$ can be chosen independently of $g$. The next result is due to Heineken.

\begin{lemma}\label{hnk}\textnormal{\textbf{\cite[12.3.1]{robinson}}}
Let $a$ be a right $n$-Engel element of a group $G$. Then $a^{-1}$ is a left $(n+1)$-Engel element of $G$.
\end{lemma}

The next two results are almost obvious. The reader can consult for example Lemma 7 and Lemma 13 in \cite{dms1} for details.

\begin{lemma}\label{powers out}
If $M$ is an abelian normal subgroup of $G$, then for all $g,h \in M$ and $a\in G$ we have $[gh, {}_n a] = [g, {}_n a][h, {}_n a]$. In particular, $[g^i, {}_n a] = [g, {}_n a]^i$ for any integer $i$.
\end{lemma}

\begin{lemma}\label{rewrite the word}
Let $w$ be the $n$-Engel word $[x,{}_n y]$. Then, there exists a group word $w_0=w_0(x_1,\dots,x_{n+1})$ such that $w = w_0(x, x^y, \dots, x^{y^{n}})$.
\end{lemma}

A proof of the following result can be found in \cite[Lemma 4.3]{dks}.

\begin{lemma}\label{w(G)' finite}
Let $G$ be a profinite group and $w$ a group word such that $|G_w| < 2^{\aleph_0}$. If $w(G)$ can be generated by finitely many $w$-values, then $w(G)'$ is finite.
\end{lemma}

Recall that the word $w$ is weakly rational if, whenever $G$ is a finite group, the set $G_w$ is closed under taking $e$-th powers for every integer $e$ coprime to $|G|$. According to \cite[Lemma 1]{gushu} this is equivalent to saying that whenever $x$ is a $w$-value of $G$, the element $x^i$ is again a $w$-value for every integer $i$ coprime to the order of $x$. Of course, $w$ is weakly rational if and only if whenever $x\in G_w$, for a finite group $G$, all generators of the cyclic subgroup $\langle x\rangle$ are contained in $G_w$. 

Throughout, $\zp$ stands for the infinite procyclic pro-$p$ group (the additive group of $p$-adic integers) and $\hz=\prod_p\zp$ for the free procyclic group. We say that an element $i\in\hz$ is a generator if $\langle i\rangle=\hz$. The following lemma is straightforward using the routine inverse limit argument.
\begin{lemma}\label{weakly rational words in profinite groups}
Let $w$ be a weakly rational word and $G$ a profinite group. Suppose that $a$ is a $w$-value in $G$. If $i$ is any generator of $\hz$, then the power $a^i$ is again a $w$-value in $G$.
\end{lemma}

Given a subgroup $K$ and an element $a$ of a group $G$, we denote by $[K,a]$ the subgroup of $G$ generated by all commutators $[g,a]$ where $g\in K$. By induction we define $[K,{}_1a]=[K,a]$ and $[K,{}_na]=[[K,{}_{n-1}a],a]$ for $n\geq2$. Of course, if $K$ is an abelian normal subgroup, then $[K,{}_na]$ coincides with all elements of the form $[g,{}_na]$ where $g\in K$, by Lemma \ref{powers out}. 

\begin{lemma}\label{[w(G), {}_n x] finite exponent}
Let $v=v(x_1,\dots,x_k)$ be any group word and $w = [x, {}_n v]$. Let $G$ be a profinite group such that $|G_w|<2^{\aleph_0}$ and $w(G)$ is generated by finitely many $w$-values. Then, for any $v$-value $a$, the subgroup $[w(G), {}_na]$ has finite exponent. 
\end{lemma}
\begin{proof}
We know from Lemma \ref{w(G)' finite} that $w(G)'$ is finite. Taking the quotient over this subgroup we can assume that $w(G)$ is abelian. Choose $g\in w(G)$. For any integer $i$, by Lemma \ref{powers out}, we have  $[g, {}_n a]^i = [g^i, {}_n a]$. So $[g, {}_n a]^i$ is a $w$-value. Since the set $G_w$ of $w$-values is closed, the procyclic subgroup generated by $[g, {}_n a]$ is contained in $G_w$, and the assumption that $|G_w| < 2^{\aleph_0}$ yields that $[g, {}_n a]$ has finite order. This happens for any $g\in w(G)$.

Now, for each positive integer $j$, define $$W_j = \{g\in w(G) \, | \, [g, {}_n a]^j = 1\}.$$ The sets $W_j$ are closed and, by the previous paragraph, cover $w(G)$. The Baire Category Theorem (\cite[p. 200]{kelley}) tells us that at least one of these sets has non-empty interior. Therefore there is an index $m$, an element $g\in w(G)$ and an open subgroup $S\leq w(G)$ such that $gS\subseteq W_m$. In particular, $[g, {}_n a]^m = 1$. Taking into account that $w(G)$ is abelian, we deduce from Lemma \ref{powers out} that, for every $s\in S$, $$1 = [gs, {}_na]^m = [g, {}_na]^m [s, {}_na]^m = [s, {}_na]^m.$$ Therefore $[S, {}_n a]^m = 1$. Let $[w(G):S] = l$. We see that $h^l$ belongs to $S$ for all $h \in w(G)$. Therefore $[w(G), {}_n a]^{lm} = 1$, as required. 
\end{proof}

\begin{corollary}\label{28}
Assume the hypothesis of Lemma \ref{[w(G), {}_n x] finite exponent}.  If $a_1,\dots,a_k$ are finitely many $v$-values, then there exists a finite normal subgroup $T$ of $G$ such that the images of $a_1,\dots,a_k$ in $G/T$ are $2n$-Engel.
\end{corollary}
\begin{proof} Assume, by Lemma \ref{w(G)' finite}, that $w(G)$ is abelian. By the previous lemma, $[w(G), {}_n a_i]$ has finite exponent. Choose a positive integer $m$ such that the exponent of $[w(G), {}_n a_i]$ divides $m$ for each $i=1,\dots,k$. Let $T$ be the subgroup of $w(G)$ generated by all elements having finite order at most $m$. Being $w(G)$ abelian and finitely generated, all of its subgroups can be generated by the same number of elements. As $T$ is finitely generated and has exponent at most $m$, we see that all finite images of it have bounded order, and we conclude that $T$ must be finite. The images of $a_1,\dots,a_k$ in $G/T$ are $2n$-Engel, as claimed. 
\end{proof}

The next lemma is closely related to \cite[Lemma 3.7]{dms3}.

\begin{lemma}\label{values are engel}
Let $w = [v, {}_n u]$, where $v$ and $u$ are group words such that in any group all $v^{-1}$-values are also $u$-values. Let $G$ be a profinite group in which $|G_w| < 2^{\aleph_0}$ and $w(G)$ is generated by finitely many $w$-values. Choose finitely many $v$-values $a_1, \dots, a_k\in G$. There exists a finite characteristic subgroup $T$ of $w(G)$ such that in $G/T$ the elements $a_1^{-1}, \dots, a_k^{-1}$ are $(2n+2)$-Engel.
\end{lemma}

\begin{proof} Recall a commutator identity $[x, y^{-1}, y^{-1}] = [y^{xy^{-1}}, y^{-1}]$ that holds in any group. Thus, if $a$ is a $v$-value in $G$, then for every $g\in G$ we have
$$[g, {}_{n+1} a^{-1}] = [a^{ga^{-1}}, {}_n a^{-1}] \in G_w.$$ Here we used the fact that $a^{-1}$ belongs to $G_u$. Hence for any $g\in G$, the element $[g, {}_{n+1} a^{-1}]$ is a $w$-value. Let $w_0$ be the group word $[y, {}_{n+1} v^{-1}]$, where $y$ is an independent variable. Note that all $w_0$-values are also $w$-values in any group. In particular, $w_0(G)$ is a subgroup of $w(G)$. Because of Lemma \ref{w(G)' finite}, without loss of generality we can assume that $w(G)$ is abelian. It follows that also $w_0(G)$ is abelian, and $|G_{w_0}| < 2^{\aleph_0}$. Thus, we mimic the proof of Lemma \ref{[w(G), {}_n x] finite exponent} with respect to the word $w_0$ and the subgroup $w(G)$ and conclude that $[w(G), {}_{n+1} a^{-1}]$ has finite exponent $m$. In particular, since $[g, {}_{n+1} a^{-1}]$ belongs to $w(G)$ for all $g \in G$, we have $[g, {}_{2n+2} a^{-1}]^{m} = 1$. Let $T_m$ be the subgroup of $w(G)$ generated by all elements of order at most $m$. Since $w(G)$ is a finitely generated abelian subgroup, it follows that $T_m$ is finite (and, of course, $T_m$ is normal in $G$). Thus, we have shown that for each $a\in G_v$ there is a finite normal subgroup, say $T_a$, contained in $w(G)$, modulo which the element $a^{-1}$ is $(2n+2)$-Engel. Note that the inverses of the given $v$-values $a_1,\dots,a_k$ are $(2n+2)$-Engel modulo the product $T = T_{a_1} T_{a_2} \dots T_{a_k}$. This completes the proof.
\end{proof}
The next result can be found in \cite[Proposition 1]{dms1} (see also \cite{stt}). 

\begin{lemma}\label{local nilpotency mcw-power}
Let $n, k, q$ be positive integers and let $v$ be a multilinear commutator word. Let $G$ be a finite group in which all $v^q$-values are $n$-Engel. If $K$ is a subgroup of $G$ generated by $k$ $v^q$-values, then $K$ is nilpotent of $(k,n,q)$-bounded class. 
\end{lemma}

The next result is straightforward from \cite[Lemma 4.1]{danilo_shumyatsky}.
\begin{lemma}\label{gruenberg}
Let $G$ be a soluble profinite group, generated by finitely many $n$-Engel elements. Then $G$ is nilpotent.
\end{lemma}

Recall that an element $a$ of a procyclic group $K$ is a generator if $\langle a\rangle=K$. 
Let $\mathbb{U}_p$ denote the set of generators of $\zp$ and $\mathbb{U}$ the set of generators of $\hz$. Of course, $\mathbb{U}_p=\zp\setminus\Phi(\zp)$, where $\Phi(\zp)$ is the Frattini subgroup of $\zp$. One can check that $\mathbb{U}=\prod_p\mathbb{U}_p$ is the Cartesian product of $\mathbb{U}_p$ over all primes $p$. It is important to note that, if we view $\hz$ as a ring, then the set $\mathbb{U}$ is precisely the multiplicative group of $\hz$.

Let $\{p_1,p_2,\dots\}$ be the set of all primes. Every element $i$ of $\hz$ can be uniquely written as an infinite product $i=\prod i_n$, where $i_n\in\mathbb Z_{p_n}$. Note that $i\in\mathbb{U}$ if and only if $i_n\in\mathbb U_{p_n}$ for each $n=1,2,\dots$. A procyclic group $K$ can be uniquely written as a Cartesian product $K=\prod_n K_n$, where $K_n$ is the pro-$p_n$ part of $K$ (here $K_n$ is either trivial or the Sylow $p_n$-subgroup of $K$). Each element $a\in K$ can be uniquely written in the form $a=\prod_n a_n$, where $a_n\in K_n$. In what follows we will use that $a^i=\prod_n a_n^{i_n}$ whenever $i\in\hz$.

Let $F=F(x_1,\dots,x_k)$ be the free profinite group of rank $k$ and let $F_j$ denote the $j$th term of the lower central series of $F$. Nontrivial elements of $F$ are called profinite words.  Since the abstract free group of rank $k$ naturally embeds in $F$, every abstract group word on $k$ variables can be viewed as a profinite word. We say that the (profinite) word $w=w(x_1,\dots,x_k)$ has degree $j$ if $w\in F_j\setminus F_{j+1}$. For completeness' sake we say that the elements lying in the intersection of the lower central series of $F$ have infinite degree. 

\begin{lemma}\label{profinite lemma 10} Let $w=w(x_1,\dots,x_k)$ be a  profinite word of degree $j$. Let $G$ be a profinite nilpotent group generated by elements $a_1,\dots,a_k$ and denote by $X$ the set of all elements of the form $w(a_1^i,\dots,a_k^i)$, where $i\in\mathbb{U}$. Assume that $|X|<2^{\aleph_0}$. Then, every element of $X$ has finite order. 
\end{lemma}
\begin{proof} Let $c$ be the nilpotency class of $G$. If $c-j$ is negative, then $X=1$ and the result holds. Therefore we assume that $c-j\geq 0$ and argue by induction on $c-j$. In the free group $F$, modulo $F_{j+1}$, the word $w$ is a product of $\gamma_j$-words in $x_1,\dots,x_k$. Therefore for any $s\in\hz$ we have $$w(x_1^s,x_2^s,\dots,x_k^s)=w(x_1,x_2,\dots,x_k)^{s^j}w_s(x_1,x_2,\dots,x_k),$$ where $w_s(x_1,x_2,\dots,x_k)$ is a profinite word of degree at least $j+1$ (see for example \cite[Lemma 1.3.3]{Segal}). Here the word $w_s$ depends on $s$. For this word, if $s\in\mathbb{U} $, there are less than $2^{\aleph_0}$ elements of the form $w_s(a_1^i,a_2^i,\dots,a_k^i)$, where $i\in\mathbb{U}$. Indeed, use the fact that the product $is$ belongs to $\mathbb{U}$. We have $$w_s(a_1^i,a_2^i,\dots,a_k^i)=(w(a_1^i,a_2^i,\dots,a_k^i)^{s^j})^{-1}w(a_1^{is},a_2^{is},\dots,a_k^{is}).$$ In  $G$ there are less than $2^{\aleph_0}$ elements of the form $w(a_1^i,a_2^i,\dots,a_k^i)$ and as many of the form $w(a_1^{is},a_2^{is},\dots,a_k^{is})$. Hence, $G$ contains less than $2^{\aleph_0}$ elements of the form $w_s(a_1^i,a_2^i,\dots,a_k^i)$. By induction, all elements of the form $w_s(a_1^i,a_2^i,\dots,a_k^i)$ have finite order.

Consider now the particular case where $G$ is a pro-$p$ group for some prime $p$. In this case the torsion part of $G$ is a finite subgroup. Passing to the quotient over that subgroup we can assume that all elements of the form $w_s(a_1^i,a_2^i,\dots,a_k^i)$ are trivial and so 
\begin{equation*}
    w(a_1^{is},\dots,a_k^{is})=w(a_1^s,\dots,a_k^s)^{i^j}=w(a_1,\dots,a_k)^{(is)^j}
\end{equation*} 
for each $i,s\in\mathbb{U}$. Since we are now in the case where $G$ is a pro-$p$ group, the above equation holds for each $i,s\in\mathbb{U}_p$. Set $a=w(a_1^s,\dots,a_k^s)$ and $K=\langle a\rangle$. Remark that the set $\mathbb{U}_p$ is closed in the profinite topology of $\zp$. In a natural way we have a continuous map $\psi$ from $\mathbb{U}_p$ to $K$ taking each $i\in\mathbb{U}_p$ to $a^{i^j}$. The cardinality of the image of $\psi$ is less than $2^{\aleph_0}$. Proposition 2.1 from \cite{dks} ensures the existence of an open subset $U$ of $\mathbb{U}_p$ such that $\psi$ is constant on $U$. Thus, for any $i_1,i_2\in U$ we have $a^{i_1^j}=a^{i_2^j}$. Since $K$ is a procyclic pro-$p$ group, we conclude that $K$ is finite.  This completes the proof in the case where $G$ is a pro-$p$ group.

We now drop the assumption that $G$ is a pro-$p$ group. Let again $\{p_1,p_2,\dots\}$ be the set of all primes and write $G=\prod_n G_{p_n}$, where $G_{p_n}$ is the pro-$p_n$ part of $G$. For each $t=1,\dots,k$ write $a_t=\prod_na_{tn}$, where $a_{tn}$ belongs to $G_{p_n}$.

Given $i=\prod i_n\in\mathbb{U}$, we have \begin{equation}
   w(a_1^i, \dots, a_k^i) = \prod_n w(a_{1n}^{i_n}, \dots, a_{kn}^{i_n}),
\end{equation} \noindent where each $i_n$ belongs to $\mathbb{U}_{p_n}$. Let $$X_n = \{w(a_{1n}^{i_n}, \dots, a_{kn}^{i_n}) \, | \, i_n \in \mathbb{U}_{p_n}\}.$$  
Recall that $\mathbb{U}=\prod_n\mathbb{U}_{p_n}$. It follows that $X=\prod_n X_n$ is the Cartesian product of $X_n$ since any element from $X$ can be decomposed as in equation (1). 

Since $|X| < 2^{\aleph_0}$, the number of nontrivial $X_n$ is finite. We already know that the lemma holds in the case of pro-$p$ groups and therefore it holds for each pro-$p$ part $G_n$ of $G$. This implies that every element in $X_n$ has finite order. Taking into account that $X=\prod_n X_n$, where only finitely many of $X_n$ are nontrivial, we deduce that every element in $X$ has finite order. The proof is complete.
\end{proof}

\section{Proof of Theorem 1.1}

Lemma \ref{profinite lemma 10} provides an important technical tool that can be used to establish finiteness of certain nilpotent subgroups. Thus, in a natural way a problem on conciseness of a word can be reduced to finding an appropriate nilpotent subgroup. The next lemma shows that under the hypotheses of Theorem \ref{theorem 1} the group $G$ has plenty of almost nilpotent subgroups.

\begin{lemma}\label{v-locally-finite-by-nilpotent}
Let $n,q$ be positive integers and $v$ a multilinear commutator word. Let $w$ be either the word $[v^q, {}_n y]$ or $[y, {}_n v^q]$, where $y$ is an independent variable. Assume that $G$ is a profinite group such that $|G_w| < 2^{\aleph_0}$ and $w(G)$ is generated by finitely many $w$-values. Let $K$ be a subgroup generated by finitely many $v^q$-values. There exists a finite normal subgroup $T$ of $G$ such that $KT/T$ is nilpotent.
\end{lemma}
\begin{proof}
Let $a_1, \dots, a_k$ be $v^q$-values in $G$ and let $K=\langle a_1, \dots, a_k\rangle$. Suppose first that $w(G)=1$. The combination of Lemma \ref{hnk} and Lemma \ref{mcw symmetric} ensures that all $v^q$-values are $(n+1)$-Engel. By Lemma \ref{local nilpotency mcw-power} the image of $K$ in any finite quotient of $G$ is nilpotent of $(k,n,q)$-bounded class. Hence, $K$ is nilpotent of $(k,n,q)$-bounded class.

We now drop the assumption that $w(G)=1$. Because of Lemma \ref{w(G)' finite}, without loss of generality we can assume that $w(G)$ is abelian. The previous paragraph shows that $K$ is soluble since $K\cap w(G)$ is abelian. Applying Lemma \ref{values are engel} in case $w = [v^q, {}_ny]$ or Corollary \ref{28} in case $w = [y, _n v^q]$, deduce that $G$ has a normal finite subgroup $T$ such that that $a_1^{-1}, \dots, a_k^{-1}$ are $2n+2$ or $2n$-Engel in $G/T$, respectively. The image of $K$ in $G/T$ is a soluble subgroup generated by $k$ Engel elements. Theorem \ref{gruenberg} then ensures that $K/(K\cap T)$ is nilpotent. The result follows. 
\end{proof}

We are in a position to prove Theorem \ref{theorem 1}, which we restate here for the reader's convenience: 
\medskip

\noindent{\it Let $k,n$ and $q$ be positive integers, and let $v=\gamma_k(x_1,x_2,\dots ,x_k)$. Suppose that $w$ is one of the words $[y, {}_n v^q]$ or $[v^q, {}_n y]$. If $G$ is a profinite group such that $|G_w| < 2^{\aleph_0}$ and $w(G)$ is generated by finitely many $w$-values, then $w(G)$ is finite.}

\begin{proof} Consider first the case where $w=[y, {}_n v^q]$. In view of Lemma \ref{w(G)' finite} we can pass to the quotient $G/w(G)'$ and assume that $w(G)$ is abelian. Let $a$ be any $v^q$-value and $t$ any element of $G$. Let $K$ be the subgroup of $G$ generated by $a$ and $a^t$. Lemma \ref{v-locally-finite-by-nilpotent} ensures that $\gamma_{c}(K)$ is contained in a finite normal subgroup $T$ of $G$ for some positive integer $c$. We pass to the quotient $G/T$ and assume that $K$ is nilpotent.  Note that we can rewrite the $w$-value $[t, {}_n a]$ as $[a^{-t}a, {}_{n-1} a]$. Let $w_0(x_1, x_2)$ be the word $[x_1^{-1}x_2, {}_{n-1} x_2]$. If $i$ belongs to $\mathbb{U}$, then $$w_0((a^{t})^i, a^i) = [(a^{-t})^ia^i, {}_{n-1} a^i]=[t, {}_n a^i].$$ Recall that the word $v^q$ is weakly rational. Therefore Lemma \ref{weakly rational words in profinite groups} ensures that the power $a^i$ must be again a $v^q$-value, and then $w_0((a^{t})^i, a^i)$ is a $w$-value for all $i \in \mathbb{U}$. If $X = \{w_0((a^{t})^i, a^i), \, i \in \mathbb{U}\}$, we must have $|X| < 2^{\aleph_0}$. Invoking Lemma \ref{profinite lemma 10}, we conclude that the element $w_0(a^{t}, a) = [t, {}_n a]$ has finite order. Since this argument holds for all $w$-values in $G$, it follows that $w(G)$ is an abelian subgroup generated by finitely many elements of finite order. We conclude that $w(G)$ is finite, as claimed.

Now we deal with the case where $w=[v^q, {}_n y]$. Once again, we assume that $w(G)$ is abelian. Let $b$ be any $v^q$-value in $G$ and $t \in G$. Consider the subgroup $L$ of $G$ generated by the $v^q$-values $b, b^t, \dots, b^{t^{n}}$. Applying Lemma \ref{v-locally-finite-by-nilpotent}, we deduce that $\gamma_{c}(K)$ is contained in a finite normal subgroup $T$ of $G$ for some $c\geq1$. Pass to the quotient $G/T$ and assume that $L$ is nilpotent. We apply Lemma \ref{rewrite the word} to the word $w$ and rewrite it as a word $w_1 = w_1(v^q, (v^q)^y, \dots, (v^q)^{y^{n}})$. If we take any $i \in \mathbb{U}$, we use Lemma \ref{weakly rational words in profinite groups} to deduce that $b^i$ is a $v^q$-value, which implies that $$w_1(b^i, (b^t)^i, \dots, (b^{t^{n}})^i) = [b^i, {}_n t]$$ is a $w$-value for all $i \in \mathbb{U}$. Once again, we may apply Lemma \ref{profinite lemma 10} and conclude that $[b, {}_n t]$ has finite order. Since this argument holds for any $w$-value of $G$, we conclude that $w(G)$ is an abelian subgroup generated by finitely many elements of finite order. Therefore $w(G)$ is finite. This completes the proof.
\end{proof}

\section{Proof of Theorem 1.2}

Recall the statement of Theorem \ref{theorem 2}: \medskip

\noindent{\it Let $n$ and $q$ be non-negative integers. Let $u$ be one of the words $y^q$, $[y_1, y_2]^q$ or any commutator-closed word. Assume that $v$ is a weakly rational word such that all $v$-values are also $u^{-1}$-values in any group and consider $w = [v, {}_n u]$. If $G$ is a profinite group such that $|G_w| < 2^{\aleph_0}$ and $w(G)$ is generated by finitely many $w$-values, then $w(G)$ is finite.}
\medskip

The following two lemmas are profinite versions of Proposition 2.1 and Proposition 2.2 of \cite{dms3}, respectively. The proofs of both are straightforward from the corresponding results on finite groups via the usual inverse limit argument.

\begin{lemma}\label{proposition 2.1}
Let $k,m, n, q, s$ be positive integers and $w$ a word. Let $G$ be a profinite group satisfying a law $w \equiv 1$ and let $H$ be a subgroup of $G$ generated by elements $a_1, \dots , a_k$ such that $$[a_i, \underbrace{[g_1, g_2]^q, \dots , [g_1, g_2]^q}_{n}, [g_1, g_2]^s] = 1$$
\noindent for each $g_1, g_2 \in H$ and each $i = 1, \dots , k$. Assume further that the elements $a_1^{-1}, \dots, a_k^{-1}$ are $m$-Engel. Then $H$ is nilpotent.
\end{lemma}

\begin{lemma}\label{proposition 2.2}
Let $k, m, n, q, s$ be positive integers, and let $u, v$ be words such that $u$ is commutator-closed and the $v^{-1}$-values are $u$-values in every group. Let $G$ be a profinite group satisfying the law $$[v, \underbrace{u^q, \dots , u^q}_{n}, u^s] \equiv 1$$ \noindent and assume that all $v^{-1}$-values are $m$-Engel in $G$. Suppose that $H$ is a subgroup of $G$ generated by finitely many $v$-values. Then $H$ is nilpotent.
\end{lemma}

\begin{lemma}\label{props 2.1 2.2 shortcut}
Let $G$, $u,v$ and $w$ be as in Theorem \ref{theorem 2}. Assume that $K$ is a subgroup generated by finitely many $v$-values. There is a finite normal subgroup $T$ of $G$ such that $KT/T$ is nilpotent.
\end{lemma}

\begin{proof}
We will prove only the case where $u = y^q$ or $[y_1,y_2]^q$, since the other case can be obtained by the same argument simply replacing the use of Lemma \ref{proposition 2.1} by that of Lemma \ref{proposition 2.2}.

So let $u = y^q$ or $[y_1, y_2]^q$. If $w(G)=1$, we use Lemma \ref{values are engel} with $T = 1$. Then Lemma \ref{proposition 2.1} applies and we conclude that $K$ is nilpotent. Suppose that $w(G)\neq1$. In view of Lemma \ref{w(G)' finite} we can pass to the quotient $G/w(G)'$ and simply assume that $w(G)$ is abelian. Then $K$ is soluble. Now we apply Lemma \ref{values are engel} again and conclude that there is a finite normal subgroup $T$ of $G$ such that the image of $K$ in $G/T$ is generated by finitely many $(2n+2)$-Engel elements. Thus, Theorem \ref{gruenberg} ensures that the image of $K$ in $G/T$ is nilpotent. 
\end{proof}

\begin{proof}[Proof of Theorem \ref{theorem 2}]
Let $a \in G_v$ and $t \in G_u$. Let $K$ be the subgroup of $G$ generated by the elements $a, a^t, \dots, a^{t^{n}}$. Lemma \ref{props 2.1 2.2 shortcut} tells us that for some integer $c$ the subgroup $\gamma_{c}(K)$ is contained in a finite normal subgroup $T$ of $G$. We can pass to the quotient $G/T$ and without loss of generality assume that $K$ is nilpotent. By Lemma \ref{rewrite the word}, there is a word $w_0=w_0(x_1,\dots,x_{n+1})$ such that  $w_0(v, v^u, \dots, v^{u^{n}})=[v, {}_n u]$. Set $X = \{w_0(a^i, (a^t)^i, \dots, (a^{t^{n}})^i), \, i \in \mathbb{U}\}$. Combining the weak rationality of $v$ with Lemma \ref{weakly rational words in profinite groups}, we conclude that $X$ is contained in $G_w$. Now Lemma \ref{profinite lemma 10} implies that all $w$-values of $G$ have finite order. Therefore $w(G)$ is generated by finitely many elements of finite order. Since $w(G)'$ is finite, the theorem follows.
\end{proof}

\end{document}